\documentclass{amsart}
\usepackage[margin=1.2in]{geometry}
\usepackage{amssymb}
\usepackage{amscd}
\usepackage[all]{xy}
\usepackage{bbm}
\usepackage{mathrsfs}
\usepackage{enumerate}
\usepackage{tikz}
\usepackage{stmaryrd}
\usepackage{etoolbox}
\usepackage{array}
\usepackage{verbatim}
\usepackage{hyperref}
\usepackage{stmaryrd}
\usepackage{bbding}

\newcommand{\N}{\mathbb{N}}
\newcommand{\Z}{\mathbb{Z}}

\newcommand{\R}{\mathbb{R}}

\newcommand{\eps}{\varepsilon}
\newcommand{\inter}{\operatorname{int}}

\newcommand{\del}{\nabla}

\newcommand{\lap}{\Delta}

\newcommand{\bd}{\partial}

\newcommand{\diam}{\operatorname{diam}}

\newcommand{\dist}{\operatorname{dist}}

\newcommand{\grad}{\del}
\newcommand{\vol}{\operatorname{Vol}}

\newcommand{\area}{\operatorname{Area}}

\newcommand{\ric}{\operatorname{Ric}}

\theoremstyle{plain}
\newtheorem{theorem}{Theorem}
 
\newtheorem{corollary}[theorem]{Corollary}
\newtheorem{prop}[theorem]{Proposition}
\newtheorem{lemma}[theorem]{Lemma}
\newtheorem{conj}[theorem]{Conjecture}

\newtheorem{question}[theorem]{Question}

\theoremstyle{definition}
\newtheorem{defn}[theorem]{Definition}
\newtheorem{rem}[theorem]{Remark}

\author{Liam Mazurowski}
\address{Department of Mathematics, Lehigh University, Bethlehem, Pennsylvania, 18015, United States}
\email{lim624@lehigh.edu}

\author{Xuan Yao}
\address{Department of Mathematics, Princeton University, Princeton, NJ 08540}
\email{xy1216@princeton.edu}

\title{Scalar curvature under weak limits of manifolds}

\begin{document}

\begin{abstract}
    We show that scalar curvature lower bounds are preserved under certain weak convergence of smooth three manifolds to a smooth limit. More precisely, suppose that $M_k$ and $M$ are smooth, closed, Riemannian three manifolds. Assume that there are smooth, surjective, $\lambda_k$-Lipschitz maps $f_k\colon M_k \to M$ and that $\vol(M_k)\to \vol(M)$ and $\lambda_k\to 1$.  Then if each $M_k$ has scalar curvature bounded below by $\kappa$ so does $M$. This result answers questions of Gromov, Sormani, Allen, and others. The proof relies on a delicate comparison between $\mu$-bubbles in $M_k$ and $\mu$-bubbles in $M$. 
\end{abstract}

\maketitle

\section{Introduction}

In Riemannian geometry, one often studies a sequence of manifolds $M_k$ converging to a limiting space $M$, and would like to  understand which geometric properties of $M_k$ are inheritied by $M$. In particular, it is an important problem to understand whether  curvature lower bounds persist under such limits. For sectional curvature and Ricci curvature, satisfying answers to this problem can be given using the theory of Alexandrov spaces \cite{burago1992ad} and CD spaces \cite{lott2009ricci,sturm2006geometry1,sturm2006geometry2} respectively. However, the situation for scalar curvature remains elusive.

In \cite{gromov2014dirac,gromov2019four}, Gromov formulated and proved that scalar curvature lower bounds are preserved under $C^0$-convergence of metric tensors.  Bamler \cite{bamler2016ricci} gave an alternative proof of this result based on Ricci flow;  also see \cite{mazurowski2026quantification} for an approach based on harmonic functions. It is natural to ask whether this convergence condition can be further weakened. Several negative results in this direction are known.  For example, scalar curvature lower bounds are not preserved under Gromov-Hausdorff convergence or intrinsic flat convergence; see \cite[Section 6.8]{gromov2019four}. Likewise,  scalar curvature lower bounds are not preserved under uniform convergence of distance functions \cite{lee2024metric}. In this paper, we give a positive result showing that scalar curvature lower bounds are preserved under certain weak convergence of manifolds. Our main theorem is the following. 

\begin{theorem}
\label{theorem:Improved}
    Let $(M_k,g_k)_{k\in \N}$ and $(M,g)$ be smooth, closed three manifolds. Assume that there are smooth, surjective, $\lambda_k$-Lipschitz  maps $f_k\colon M_k \to M$ and suppose that $\vol(M_k)\to \vol(M)$ and $\lambda_k\to 1$. If all the manifolds $M_k$ have scalar curvature at least $\kappa$ then so does $M$. 
\end{theorem}

This theorem answers questions of Gromov \cite{gromov2019four}, Sormani et al.\ \cite{sormani2023conjectures}, and Allen \cite{allen2024oberwolfach}. See the next section for further discussion. 

\subsection{Discussion} There are many problems and conjectures regarding the preservation of scalar curvature lower bounds.   We refer to \cite{gromov2019four} and \cite{sormani2023conjectures} for nice surveys on this topic. 
As mentioned above, scalar curvature lower bounds are not necessarily preserved under intrinsic flat convergence.  In  four lectures on scalar curvature, Gromov \cite[Section 6.8]{gromov2019four} proposed the following stronger forms of intrinsic flat convergence.

\begin{defn}
    A Riemannian $(\alpha,\beta)$-cobordism between closed oriented Riemannian $n$-manifolds $(X_1,g_1)$ and $(X_2,g_2)$ is an oriented Riemannian $(n+1)$-manifold $(W,g)$, 
with oriented boundary $\partial W=X_1-X_2$, such that
\begin{gather*}
d_{X_i}(p,q) = d_W(p,q), \quad \text{for all } p,q\in X_i
\end{gather*}
and
\[
\dist_{\text{Hau}}(X_1,X_2)\leq \alpha,\quad \vol(W)\leq \beta.
\]
We write $W_{\alpha,\beta}\colon X_1\to X_2$ to indicate such a cobordism. 
\end{defn}

\begin{defn}
A Riemannian $(\alpha,\beta,\lambda)$-cobordism between $X_1$ and $X_2$ is an $(\alpha,\beta)$-cobordism with a $\lambda$-Lipschitz retraction $W\to X_2\subset W$. Again we write $W_{\alpha,\beta,\lambda}\colon X_1\to X_2$ to indicate such a cobordism. 
\end{defn}

Gromov \cite[Section 6.8]{gromov2019four}  conjectured that convergence via $(\alpha,\beta)$- or $(\alpha,\beta,\lambda)$-cobordisms  should preserve scalar curvature lower bounds. 

\begin{conj}[Gromov, $(\alpha,\beta)$-convergence]\label{conjecture: Gromov} 
   Assume there is a sequence of $(\alpha_i,\beta_i)$-cobordisms $W_{\alpha_i,\beta_i}\colon X_i \to X$ with $\alpha_i,\beta_i\to 0$.  If each $X_i$ has non-negative scalar curvature then so does $X$.  

\end{conj}

\begin{conj}[Gromov, $(\alpha,\beta,\lambda)$-convergence]\label{conj: gromov alph beta lambda}
 Assume there is a sequence of $(\alpha_i,\beta_i,\lambda_i)$-cobordisms $W_{\alpha_i,\beta_i,\lambda_i}\colon X_i \to X$ with $\beta_i \to 0$ and $\lambda_i \to 1$. If each $X_i$ has non-negative scalar curvature then so does $X$.  
\end{conj}

\begin{rem}
    Gromov also notes that ``possibly the validity of these conjectures needs additional conditions on $X_i$, e.g. the convergence of the volumes $\vol(X_i)\to \vol(X)$.''
\end{rem}

Assume that $W_{\alpha_i,\beta_i,\lambda_i}\colon X_i\to X$ is a sequence of $(\alpha_i,\beta_i,\lambda_i)$ cobordisms with $\beta_i\to 0$ and $\lambda_i\to 1$. We note that in this setting the retraction maps $X_i\to X$ must have degree one. Therefore, if we assume in addition that $\vol(X_i)\to \vol(X)$ then all the assumptions of Theorem \ref{theorem:Improved} are satisfied.  Hence Theorem \ref{theorem:Improved} implies that Conjecture \ref{conj: gromov alph beta lambda} is true under the extra volume convergence assumption suggested by Gromov.

In \cite{allen2024volume}, Allen-Perales-Sormani proposed a  another strengthening of intrinsic flat convergence, namely the volume above distance below convergence.  This volume above distance below convergence has proven very useful for studying geometric stability problems; see for example \cite{allen2021almost},\cite{allen2025stability},\cite{cabrera2020stability}. 

\begin{defn} Let $M$ be a closed manifold and let $(g_k)_{k\in \N}$ and $g$ be metrics on $M$. Assume that \begin{itemize}
    \item[(i)] $g_k \ge g$,
    \item[(ii)] there is a constant $D>0$ such that $\diam(M,g_k)\le D$ for all $k$, and 
\item[(iii)] $\vol(M,g_k)\to \vol(M,g)$. 
\end{itemize} 
Then $(M,g_k)$ converges to $(M,g)$ in the volume above distance below sense. 
\end{defn}

Sormani et al.\ \cite[Question 4.2 and Remark 4.5]{sormani2023conjectures} asked whether scalar curvature lower bounds are preserved under the volume above distance below convergence.  This question was later reiterated by Allen \cite[Question 6.3]{allen2024oberwolfach}.  

\begin{question}
[Sormani et al., Allen] \label{question:sormani} Let $M$ be a closed three manifold. Assume that $(M,g_k)$ converges to $(M,g)$ in the volume above distance below sense. If each manifold $(M,g_k)$ has non-negative scalar curvature, does $(M,g)$ also have non-negative scalar curvature?  
\end{question}

Theorem \ref{theorem:Improved} shows that the answer to Question \ref{question:sormani} is yes. In fact, the diameter bound (ii) plays no role in this result. More importantly, we can also allow for convergence $M_k \to M$ where the manifolds $M_k$ are not diffeomorphic to $M$.

\subsection{Sketch of Proof}  The proof of Theorem \ref{theorem:Improved} relies on a delicate comparison between $\mu$-bubbles in $M$ and $\mu$-bubbles in $M_k$. One proceeds by contradiction. Suppose that each $M_k$ has scalar curvature at least $\kappa$, but $M$ has  scalar curvature less than $\kappa$ at some point $p$.  We  select a very small geodesic ball $\Omega_0 =  B(p,r)$ in $M$. We assume $r$ is small enough that $\Sigma := \bd \Omega_0$ is nearly round and every point in $\Omega_0$ has scalar curvature less than $\kappa$ in the $g$ metric.  

Next, we define a function $h$ localized near $p$ so that $\Omega_0$ is the unique minimizing $\mu$-bubble for the $\mathcal A^h$ functional in $M$. Then we lift $h$ to a function $h_k$ on $M_k$ using $f_k$, and we show that there are minimizing $\mu$-bubbles $\Omega_k$ in $M_k$ for the $\mathcal A^{h_k}$ functional.  We can arrange that 
\begin{equation}
\label{intro-eq4}
\mathcal A^{h_k}(\Omega_k)\to \mathcal A^h(\Omega_0), \quad \text{as } k\to \infty
\end{equation}
by choosing $r$ and $h$ carefully.
This step relies on both the convergence of the volumes and the fact that $f_k$ is $\lambda_k$-Lipschitz. 

Since $\Omega_0$ is the unique minimizer in $M$, one can exploit \eqref{intro-eq4} to show that there is a connected component $\Gamma_k$ of $\bd \Omega_k$ such that the projection of $\Gamma_k$ to $M$  converges to $\Sigma$.  We note that this convergence takes place in the weak varifold topology, and in principle the projection of $\Gamma_k$ to $M$ may have long thin fingers that extend far from $\Sigma$. This type of behavior cannot happen under the stronger assumption of $C^0$-convergence of metric tensors.  

Nevertheless, the weak varifold convergence is enough to compare the stability inequality on $\Gamma_k$ with the stability inequality on $\Sigma$. More precisely, since $\Sigma$ is nearly round, the stability inequality implies that 
\begin{equation}\label{intro-eq1}
\int_{\Sigma} \frac{R_g}{2} + \frac 3 4 h^2 - \vert \grad h\vert \, da \approx 4\pi.  
\end{equation}
On the other hand, the stability inequality on $\Gamma_k$ gives 
\begin{equation}
    \label{intro-eq2}
\int_{\Gamma_k} \frac{R_{g_k}}{2} + \frac 3 4 h_k^2 + \nu_k(h_k)\, da_k \le 4\pi.  
\end{equation}
Again, assuming $h$ is well-chosen, we can use the varifold convergence to show that 
\begin{equation}
\label{intro-eq3}
\liminf \int_{\Gamma_k} \frac 3 4 h_k^2 + \nu_k(h_k)\, da_k \ge \int_{\Sigma} \frac 3 4 h^2 - \vert \grad h\vert\, da. 
\end{equation}
Finally, we reach a contradiction since equations \eqref{intro-eq1}, \eqref{intro-eq2}, \eqref{intro-eq3} taken together are incompatible with the assumption $R_{g_k}\ge \kappa$ and $R_g < \kappa$. 

\begin{question}
    The proof of Theorem \ref{theorem:Improved} can be specialized to give a $\mu$-bubble proof of Gromov's $C^0$-convergence theorem in dimension three. Is there an inductive $\mu$-bubble proof of the $C^0$-convergence theorem in higher dimensions?
\end{question} 

\begin{question} Along these lines, it would be interesting to know if there is a spectral version of the $C^0$-convergence theorem. For example, if $g_k$ converges to $g$ in $C^0$, does this imply that $\lambda_1(-\lap_g + \frac{R_g}{2}) \ge \limsup \lambda_1(-\lap_{g_k} + \frac{R_{g_k}}{2})$?
\end{question}

\subsection{Organization} The remainder of the paper is organized as follows. First, in Section \ref{section:prelim} we discuss some background material on geometric measure theory, Caccioppoli sets, and stable $\mu$-bubbles. Then in Section \ref{section:main} we give the proof of Theorem \ref{theorem:Improved}.

\subsection{Acknowledgments}
We would like to thank Fernando Cod\'{a} Marques for his suggestion to consider a $\mu$-bubble proof of the $C^0$-convergence theorem. X.Y.\ would also like to thank Brian Allen for an interesting conversation on this topic.

\section{Preliminaries}
\label{section:prelim}

In this section, we discuss some preliminaries on geometric measure theory, Caccioppoli sets, and stable $\mu$-bubbles.  

\subsection{Geometric Measure Theory}

We start by recalling the area and co-area formulas. 

\begin{prop}[Area Formula]
Suppose $f: M\to N$ is a Lipschitz map and $\dim M\leq \dim N$. Then for any measurable subset $A\subset M$ and a bounded, measurable function $h$ on $N$ we have
\[
\int_{A}(h\circ f)J_f\, dv_{M}=\int_{N}h(x)\theta_A(x)\, dv_{N},
\]
where 
\[
J_f(y)=\sqrt{\det(df^{*}(y)\circ df(y))}
\]
is the Jacobian of $f$, $df^*$ is the adjoint transformation of $df$,
and
\[
\theta_A(x):=\#\{y\in A: f(y)=x\}
\]
is the multiplicity function.
\end{prop}

\begin{prop}[Co-area Formula]
Suppose $f: M\to\mathbb R$ is a real valued Lipschitz function. Then for any $L^1$-integrable function $g$ on $M$, we have
\[
\int_{M}g(x)|\nabla f(x)|\, dv_g=\int_{\mathbb R}\int_{f^{-1}(t)}g(x)\, da_g\, dt.
\]
\end{prop}

Next, we discuss currents and varifolds.  Although we will primarily work with currents and varifolds associated to smooth submanifolds, the language of geometric measure theory is quite useful for discussing convergence and limiting behavior.  It is also the natural setting for defining the pushforward associated to a Lipschitz map, which we will need repeatedly later. 

\newcommand{\M}{\mathbf M}

\subsection{Currents} Let $M^3$ be a closed Riemannian manifold.  For $k=2,3$, we denote by $I_k(M,\Z_2)$ the space of flat chains in $M$ with coefficients in $\Z_2$ for which $\mathbf M(T) + \mathbf M(\bd T) < \infty$.  Here $\mathbf M$ denotes the mass functional, and we use $\|T\|$ to denote the mass measure associated to $T$.   Let $\mathcal Z_2(M,\Z_2)$ denote the set of $T\in I_2(M,\Z_2)$ for which $T = \bd S$ for some $S\in I_3(M,\Z_2)$.   For $k=2$, we will only work with elements in $\mathcal Z_2(M,\Z_2)$.  

We equip the spaces $I_3(M,\Z_2)$ and $\mathcal Z_2(M,\Z_2)$ with the flat topology.  In the top dimension, this is equivalent to the mass topology, and we have 
\[
S_k \to S \text{ in } I_3(M,\Z_2) \iff \M(S_k-S)\to 0. 
\]
For two dimensional cycles, the flat topology satisfies 
\[
T_k\to T \text{ in } \mathcal Z_2(M,\Z_2) \iff \text{ there are $S_k \in I_3(M,\Z_2)$ with $T_k - T = \bd S_k$ and $\M(S_k) \to 0$}. 
\]
The mass is lower semi-continuous in the flat topology. 

\begin{prop}
If $T_k \to T$ in $\mathcal Z_2(M,\Z_2)$ then $\M(T) \le \liminf \M(T_k)$. 
\end{prop}

The space $\mathcal Z_2(M,\Z_2)$ has nice compactness properties. 

\begin{theorem}
Assume that $T_k \in \mathcal Z_2(M,\Z_2)$ and that $\M(T_k)$ is uniformly bounded. Then some subsequence of $T_k$ converges to a limit in $\mathcal Z_2(M,\Z_2)$. 
\end{theorem} 

A Lipschitz map $f\colon M\to N$ naturally induces a pushforward map $f_\sharp\colon  I_k(M,\Z_2) \to I_k(N,\Z_2)$ and this map satisfies the estimate 
\begin{equation}
\label{eq:lip}
\M(f_\sharp T) \le \operatorname{Lip}(f)^k \M(T). 
\end{equation} 
We also note that pushforward commutes with boundary: $\bd (f_\sharp T) = f_\sharp(\bd T)$.  

\newcommand{\lb}{\llbracket}
\newcommand{\rb}{\rrbracket}

Every smooth submanifold $\Sigma^k$ of $M$ determines an associated element $\lb \Sigma\rb \in I_k(M,\Z_2)$.   This element satisfies $\M(\lb\Sigma\rb) = \mathcal H^k(\Sigma)$ where $\mathcal H^k$ is the $k$-dimensional Hausdorff measure on $M$.  We also have 
\[
\| \, \lb\Sigma\rb\, \|(U) = \mathcal H^k(\Sigma \cap U)
\]
for every measurable $U\subset M$. 
Finally, we recall the following constancy theorem. 

\begin{theorem}
Assume that $\Sigma^2$ is a smooth, closed submanifold of $M$. Assume that $T\in \mathcal Z_2(M,\Z_2)$ has support contained in $\Sigma$. Then either $T = 0$ or $T = \lb\Sigma\rb$. 
\end{theorem}

\subsection{Varifolds} Let $M^3$ be a closed Riemannian manifold.  We let $\mathcal V(M)$ denote the space of two-dimensional varifolds in $M$, i.e., Radon measures on the Grassmannian $\operatorname{Gr}_2(M)$.  We write $\|V\|$ for the mass measure associated to a varifold $V$.   Convergence $V_k\to V$ is determined by the natural topology on the space of Radon measures. In other words, we have  
\[
V_k \to V \iff V_k(f) \to V(f)
\]
for all continuous, compactly supported functions $f$ on $\operatorname{Gr}_2(M)$. 
We note that convergence of varifolds preserves the mass: if $V_k \to V$ then $\|V_k\|(M) \to \|V\|(M)$. Again one has a compactness theorem. 

\begin{theorem}
Assume that $V_k \in \mathcal V(M)$ and $\|V_k\|(M)$ is uniformly bounded. Then some subsequence converges to a limit in $\mathcal V(M)$. 
\end{theorem} 

We note that any $T \in \mathcal Z_2(M,\Z_2)$ determines a varifold which we denote $\vert T\vert$. If $T_k \to T$ and $\vert T_k\vert \to V$ then one has an inequality $\|T\| \le \|V\|$. A smooth submanifold $\Sigma^2$ in $M$ also determines a varifold which we denote by $\vert \Sigma\vert$. We note that $\vert \Sigma\vert = \vert \, \lb\Sigma\rb \, \vert$. In this case, the mass measure $\|\, \vert \Sigma\vert\, \|$ is just the Hausdorff measure $\mathcal H^2$ restricted to $\Sigma$. 

Finally, we note the following criterion for varifold convergence; see for example \cite[Section 2.1 (18)(f)]{pitts2014existence}. 

\begin{prop}
\label{prop:area-converges}
    Assume that $T_k\to T$ in $\mathcal Z_2(M,\Z_2)$. Then $\vert T_k\vert \to \vert T\vert$ as varifolds if and only if $\M(T_k)\to \M(T)$. 
\end{prop}

\subsection{Caccioppoli Sets}

Let $M$ be a smooth, closed Riemannian manifold.  We let $\mathcal C(M)$ be the collection of all Caccioppoli sets in $M$, i.e., subsets $\Omega$ of $M$ such that the characteristic function $\chi_\Omega$ has bounded variation.  There is a natural correspondence between Caccioppoli sets and elements of $I_3(M,\Z_2)$. We write $\lb \Omega\rb $ for the element in $I_3(M,\Z_2)$ determined by this correspondence, and note that $\mathcal H^2(\bd^* \Omega) = \M(\bd \lb \Omega\rb )$ where $\bd^*\Omega$ denotes the reduced boundary of $\Omega$.  We will also write $\vert \bd \Omega\vert$ for the varifold induced by $\bd \lb \Omega \rb$. 

\begin{prop}
    Assume that $f\colon M\to N$ is a smooth map between closed Riemannian manifolds of the same dimension.  Let $\Omega$ be a Caccioppoli set in $M$. Then  
    \[
    \theta_\Omega(x) := \#\{y\in \Omega: f(y)=x\}
    \]
    is a measurable function on $N$ taking integer values almost-everywhere, and the set 
    \[
    f_\sharp \Omega := \{x\in N: \theta_\Omega(x) \text{ is odd}\} 
    \]
    is a Caccioppoli set in $N$. We call $f_\sharp \Omega$ the pushforward of $\Omega$ by the map $f$. We further have that $f_\sharp \lb \Omega\rb  = \lb f_\sharp \Omega\rb $ in $I_3(N,\Z_2)$. 
\end{prop}

\begin{proof}
This follows from the pushforward formula for $f_\sharp\lb \Omega\rb $ (see for example \cite[Section 2.3]{white2009currents}) and the correspondence between Caccioppoli sets and flat chains mod 2.  
\end{proof} 

\begin{prop}
\label{prop:area-ineq}
Assume that $f\colon M \to N$ is 1-Lipschitz. Then we have $\mathcal H^2_N( \bd^* (f_\sharp \Omega)) \le \mathcal H^2_M(\bd^* \Omega)$. 
\end{prop}

\begin{proof}
By the correspondence between Caccioppoli sets and flat chains mod 2, it is equivalent to show that 
$
\M(\bd \lb f_\sharp \Omega\rb ) \le \M(\bd \lb \Omega\rb). 
$
Since 
\[
\M(\bd \lb f_\sharp \Omega\rb) = \M(\bd f_\sharp\lb \Omega\rb) = \M(f_\sharp \bd \lb \Omega\rb), 
\]
it is equivalent to show that 
\[
\M(f_\sharp \bd \lb\Omega\rb) \le \M(\bd \lb\Omega\rb). 
\]
But this follows from \eqref{eq:lip} since $f$ is 1-Lipschitz. 
\end{proof}

\begin{prop}
Assume that $f \colon M\to N$ is a Lipschitz map between manifolds of the same dimension. Let $h$ be a bounded function on $N$. Then for any $\Omega \in \mathcal C(M)$ we have
\begin{align*}
&\int_{M} \chi_{\Omega} (h\circ f) \, dv_{M} - \int_{N} \chi_{f_\sharp \Omega} h\, dv_{N} \\
&\qquad = \int_\Omega  (1-J_f) (h\circ f) \, dv_{M} + \int_{N}  (\theta_\Omega - \overline \theta_\Omega )h \, dv_{N}
\end{align*} 
where $\theta_\Omega(x)$ is as above and $\overline \theta_\Omega(x)$ is the mod 2 reduction of $\theta_\Omega$. 
\end{prop}

\begin{proof}
The area formula implies that 
\[
\int_\Omega (h\circ f)\, J_f \, dv_M = \int_{f(\Omega)} \theta_\Omega h \, dv_N. 
\]
Also we have $\chi_{f_\sharp \Omega} = \overline \theta$.  The result follows by combining these observations. 
\end{proof}

\begin{prop}
\label{prop:integral-formula}
Let $f\colon M \to N$ be a Lipschitz map between two closed manifolds of the same dimension.  Let $h$ be a smooth function defined on an open subset $U$ of $N$. Let $\Omega$ and $\Omega_0$ be two Caccioppoli sets in $M$. Assume that there is a compact subset $K\subset U$ so that the support of $f_\sharp \Omega \operatorname{\Delta} f_\sharp \Omega_0$ is contained in $K$ and the support of $\Omega \operatorname{\Delta} \Omega_0$ is contained in $f^{-1}(K)$. 
Then we have 
\begin{align*}
&\int_{M} (\chi_\Omega - \chi_{\Omega_0})(h\circ f)\, dv_{M} - \int_{N} (\chi_{f_\sharp \Omega} - \chi_{f_\sharp \Omega_0})h\, dv_{N} \\
&\quad = \int_{M} (\chi_\Omega - \chi_{\Omega_0})(h\circ f) (1-J_f)\, dv_{M} + \int_{N} ((\theta_{\Omega} -  \theta_{\Omega_0}) - (\overline \theta_{\Omega}-\overline \theta_{\Omega_0}))h\, dv_{N}. 
\end{align*}
\end{prop} 

\begin{proof}
This does not immediately follow from the previous proposition since $h$ may not be defined on all of $N$. However, our assumptions ensure that 
\[
\chi_{\Omega} - \chi_{\Omega_0} = \chi_{\Omega\cap f^{-1}(K)} - \chi_{\Omega_0\cap f^{-1}(K)}.
\]
Applying the area formula, we have 
\begin{gather*}
\int_{\Omega \cap f^{-1}(K)} (h\circ f)\, J_f\, dv_M = \int_{N} \theta_{\Omega\cap f^{-1}(K)} h\, dv_N,\\
\int_{\Omega_0 \cap f^{-1}(K)} (h\circ f)\, J_f\, dv_M = \int_{N} \theta_{\Omega_0\cap f^{-1}(K)} h\, dv_N.
\end{gather*}
Thus we have 
\begin{align*}
&\int_M (\chi_\Omega - \chi_{\Omega_0}) (h\circ f)(1-J_f)\, dv_M \\
&\qquad = \int_M (\chi_\Omega - \chi_{\Omega_0}) (h\circ f)\, dv_M + \int_N (\theta_{\Omega\cap f^{-1}(K)} - \theta_{\Omega_0\cap f^{-1}(K)}) h \, dv_N\\
&\qquad = \int_M (\chi_\Omega - \chi_{\Omega_0}) (h\circ f)\, dv_M + \int_N (\theta_{\Omega} - \theta_{\Omega_0}) h \, dv_N
\end{align*} 
since the fact that $\Omega \operatorname{\Delta} \Omega_0$ is supported in $f^{-1}(K)$ also implies that 
\[
\theta_{\Omega\cap f^{-1}(K)} - \theta_{\Omega_0\cap f^{-1}(K)} = \theta_{\Omega} - \theta_{\Omega_0}
\]
almost everywhere.  Finally, we have $\chi_{f_\sharp \Omega} - \chi_{f_\sharp \Omega_0} = \overline \theta_\Omega - \overline \theta_{\Omega_0}$ so that 
\[
\int_N (\chi_{f_\sharp \Omega} - \chi_{f_\sharp \Omega_0})h\, dv_N = \int_N (\overline \theta_\Omega - \overline \theta_{\Omega_0})h\, dv_N. 
\]
The result now follows by combining the previous equations. 
\end{proof} 

\begin{prop}
\label{prop:multiplicity-bound}
Let $f\colon M\to N$ be a smooth map. Let $\Omega$ be a Caccioppoli set in $M$. Then 
\[
\theta_\Omega(x) - \overline \theta_\Omega(x) \le \theta_M(x) - \overline \theta_M(x)
\]
for almost every $x\in N$.   
\end{prop}

\begin{proof}
It suffices to note that $j \mapsto j - (j\text{ mod } 2)$ is a non-decreasing function of $j\in \N \cup\{0\}$. 
\end{proof}

\subsection{Stable \texorpdfstring{$\mu$}{mu}-Bubbles}

Consider a Riemannian manifold $(M^3,g)$ with boundary $\partial M=\partial_-M\sqcup \partial_+ M$, where neither of $\partial_{\pm}M$ is empty. 
Suppose $h$ is a smooth function defined on the interior of $M$ so that $h\to \pm\infty$ as $x\to \partial_{\pm}M$ respectively. 
Given a Caccioppoli set $\Omega_0\subset M$ with smooth boundary and containing  $\partial_{+}M$, consider the $\mu$-bubble functional:
\begin{align*}
    \mathcal A_g^h(\Omega):=\mathcal H^2_g(\partial^*\Omega)-\int_{M}(\chi_{\Omega}-\chi_{\Omega_0})h\, dv_g,
\end{align*}
for all Caccioppoli sets $\Omega\subset M$ with  $\Omega\operatorname{\Delta}\Omega_0\subset\subset \inter(M)$. Here $\bd^*\Omega\subset \inter (M)$ is the reduced boundary of $\Omega$ in $\inter (M)$. We call an $\mathcal A^h_g$-minimizer $\Omega$ in this class a {\em $\mu$-bubble}.

The existence  of such a minimizer among all Caccioppoli sets follows from work of Gromov \cite{gromov2019four} and Zhu \cite{zhu2021width}. Note that  $\mu$-bubbles are also known as prescribed mean curvature surfaces, and one can refer to \cite{zhou2019cmc,zhou2020pmc} for a min-max construction.

\begin{prop}[Gromov \cite{gromov2019four}, Zhu \cite{zhu2021width}]\label{prop: mu-bubble regularity}
There exists a smooth minimizer $\Omega$ for $\mathcal A$ such that $\Omega\operatorname{\Delta}\Omega_0\subset\subset\inter(M). $
\end{prop}

\begin{prop}[Second variation]
    Suppose $\Omega\in\mathcal C(M)$ is a minimizer of the $\mu$-bubble functional $\mathcal A_g^h$. Then $\Sigma=\partial\Omega$ satisfies
    \[
    H_{\Sigma}=h,
    \]
    and for any $\psi\in C^{\infty}(\Sigma)$ we have
    \[
    \int_{\Sigma}-\psi\Delta_{\Sigma}\psi-(\ric(\nu,\nu)+\|A_{\Sigma}\|^2)\psi^2-\langle \nabla h,\nu\rangle\psi^2\, da \geq 0,
    \]
    where $\nu$ is the unit normal of $\Sigma$ pointing out of $\Omega$.
\end{prop}

We note that applying the stability inequality with the test function $\psi \equiv 1$ and using the Schoen-Yau rearrangement together with the Gauss-Bonnet theorem gives 
\[
\int_{\Sigma} \frac{R_M}{2} + \frac{\| \mathring A_\Sigma\|^2}{2} + \frac 3 4 h^2 + \nu(h)\, da \le 4\pi
\]
provided $\Sigma$ is connected.

\section{Proof of Main Theorem}
\label{section:main}

In this section we prove our main theorem. 

\begin{theorem}
    Let $(M_k,g_k)_{k\in \N}$ and $(M,g)$ be smooth, closed three manifolds. Assume that there are smooth, surjective, $\lambda_k$-Lipschitz maps $f_k\colon M_k \to M$ and suppose that $\vol(M_k)\to \vol(M)$ and $\lambda_k\to 1$. If all the manifolds $M_k$ have scalar curvature at least $\kappa$ then so does $M$.  
\end{theorem}

\begin{rem}
    It is easy to see that it suffices to prove the theorem assuming that $\lambda_k = 1$. From here on we will assume that $\lambda_k = 1$ for all $k$. 
\end{rem}

We now make some  preliminary observations.  First we show that the measure of the set of points in $M$ with more than one preimage under $f_k$ goes to 0 as $k\to \infty$. 
Define multiplicity functions $\theta_{M_k}\colon M\to \N$ by $\theta_{M_k}(x) = \# f_k^{-1}(x)$.  In other words, $\theta_{M_k}(x)$ is the number of preimages of $x$ under the map $f_k$. Note that $\theta_{M_k}(x) \ge 1$ for all $k$ and $x$ since we have assumed the maps $f_k$ are surjective. 
Define the sets $P_k = \{x\in M: \theta_{M_k}(x) = 1\}$. 

\begin{prop}
\label{prop:multiplicity}
    We have $\vol_g(P_k)\to \vol_g(M)$ as $k\to \infty$. 
\end{prop}

\begin{proof}
    The area formula says that 
    \[
    \int_{M_k} J_k(y)\, dv_k(y) = \int_M \theta_{M_k}(x)\, dv(x),
    \]
    where $J_k$ is the Jacobian of $f_k$. Since $f_k$ is 1-Lipschitz, we have $J_k \le 1$ and it follows that 
    \[
    \int_M \theta_{M_k} \, dv \le \vol_{g_k}(M_k).
    \]
    Since $\theta_{M_k} \ge 2$ on $M-P_k$, this implies that 
    \[
    \vol_g(P_k) + 2\vol_g(M-P_k) \le \vol_{g_k}(M_k).
    \]
    Re-arranging, this says that 
    \[
    \vol_g(M-P_k) \le \vol_{g_k}(M_k)-\vol_g(M). 
    \]
    The right hand side goes to 0 by assumption and therefore we have 
    $
    \vol_g(P_k)\to \vol_g(M).
    $
\end{proof}

This has the following immediate corollary. 

\begin{corollary}
    The degree of $f_k$ is one modulo two for large $k$. In particular, almost every point $x\in M$ has an odd number of preimages under $f_k$. 
\end{corollary}

Finally, we note that the volume convergence also holds on any subset of $M$. 

\begin{prop}
\label{prop:subset-volume}
    Let $W$ be a measurable subset of $M$ and let $W_k = f_k^{-1}(W)\subset M_k$. Then $\vol(W_k)\to \vol(W)$ as $k\to \infty$. 
\end{prop}

\begin{proof}
    The area formula implies that 
    \[
\int_{W_k} J_k(y)\, dv_k(y) = \int_{W} \theta_{W_k}(x)\, dx
    \]
    where $J_k$ is the Jacobian of $f_k$ and $\theta_{W_k}$ is the multiplicity function. 
    Since $f_k$ is 1-Lipschitz and surjective, we have $J_k \le 1$ and $\theta_{W_k}\ge 1$ and thus 
    \[
    \vol_g(W) \le \vol_{g_k}(W_k). 
    \]
Note  that 
     $f_k \colon (M_k - W_k) \to (M-W)$ must be surjective. Thus we likewise obtain 
    \[
    \vol_{g}(M-W) \le \vol_{g_k}(M_k -W_k). 
    \]
    Combining these observations we see that \begin{align*}
    0 &\le \left[\vol_{g_k}(W_k) -\vol_g(W)\right] + \left[ \vol_{g_k}(M_k-W_k) - \vol_g(M-W)\right]
    \\
    &= \vol_{g_k}(M_k) - \vol_g(M) \to 0. 
    \end{align*}
    Since both terms in brackets are non-negative, they must both go to 0. 
\end{proof}

We now proceed with the proof of Theorem \ref{theorem:Improved} in earnest. The proof of Theorem \ref{theorem:Improved} is by contradiction.  Suppose that each metric $g_k$ has scalar curvature at least $\kappa$ but there is some point $p\in M$ where the scalar curvature of $g$ is less than $\kappa$.  We need to study the geometry of small geodesic spheres centered at $p$. 

\subsection{Geometry of Small Spheres} 
Let $x$ be a geodesic normal coordinate system for $g$ centered at $p$.  Let $S_r$ be a geodesic sphere of radius $r$ centered at $p$ in the $g$ metric.   We would like to understand the expansion of various geometric quantities on $S_{r} = \{\vert x\vert = r\}$ as $r\to 0$. 

\begin{prop}
The trace free second fundamental form of $S_r$ satisfies $\vert \mathring A\vert^2 = O(r^2)$. 
\end{prop}

\begin{proof}
Let $r$ be the distance function to $p$.  Then we have 
\[
\operatorname{Hess} r = \frac{g_{S_r}}{r} + O(r) 
\]
where $g_{S_r} = g - dr^2$ 
is the induced metric on $S_r$. Now $\operatorname{Hess} r$ restricted to $S_r$ is exactly the second fundamental form $A$.  Taking the trace with respect to $g_{S_r}$, we see that 
\[
H = \frac 2 r + O(r). 
\]
The trace free second fundamental form therefore satisfies 
\[
\mathring A = A - \frac{H}{2}g_{S_r} = O(r)
\]
and it follows that $\vert \mathring A\vert^2 = O(r^2)$. 
\end{proof} 

\begin{prop}
Let $H(r,\theta)$ be the mean curvature of $S_r$ at the point $q= (r,\theta)$. Then  
\begin{gather*}
H = \frac 2 r - \frac{r}{3}\ric(\theta,\theta) + O(r^2),\\
\bd_r H = -\frac 2{r^2} + O(1),\quad \vert \grad^\top H\vert = O(1),
\end{gather*}
where $\grad^\top H$ is the projection of $\grad H$ to $T_q S_r$. 
\end{prop}

\begin{proof}
The expansion of the mean curvature is well-known. The formulas for $\bd_r H$ and $\vert \grad^\top H\vert$ then follow by taking derivatives.
\end{proof}

It is convenient to zoom in near the point $p$.  Fix some small $r > 0$ and define the rescaled metric $g_r = r^{-2}g$.  Let $\tilde x$ be the geodesic normal coordinates centered at $p$ in the metric $g_r$.   Let $\Gamma = \{\vert \tilde x\vert = \tau\}$ for some $\tau \in (1,3)$.  Then from the above expansions, we see that in the $g_r$ metric $\Gamma$ satisfies 
\begin{gather*}
\vert \mathring A_\Gamma\vert_{g_r}^2 
= O(r^4),\\
H^{g_r}_\Gamma 
= \frac{2}{\tau} + O(r^2),\\
(\bd_\nu H)_{g_r} = - \frac 2 {\tau^2} + O(r^2),\quad (\vert \grad^\top H\vert)_{g_r} = O(r^2). 
\end{gather*} 
We note that this implies 
\begin{align*} 
\vert \grad H\vert + \bd_\nu H & = -\bd_\nu H\left( \sqrt{1 + \frac{\vert \grad^\top H\vert^2}{(\bd_\nu H)^2}} - 1\right) \\
&= O(1) \left(\sqrt{1 + O(r^4)} - 1\right) = O(r^4) 
\end{align*} 
in the $g_r$ metric. 

Likewise, for a fixed small $r > 0$, we define rescaled metrics $g_{k,r} = r^{-2} g_k$.  
Then the maps $f_k\colon(M_k,g_{k,r})\to (M,g_r)$ are still $1$-Lipschitz and we have $\vol(M_k,g_{k,r})\to \vol(M,g_r)$ as $k\to \infty$. 
By our contradiction assumption, we can find $c > 0$ and a small neighborhood $U$ of $p$ in the $g$-metric so that 
\[
R_g(q) \le \kappa-c
\]
for all $q\in U$. 
The scalar curvatures of the rescaled metrics satisfy 
\[
R_{g_r} = r^2 R_g, \quad R_{g_{k,r}} = r^2 R_{g_k} \ge r^2\kappa. 
\]
In particular, assuming $r$ is small enough, we have 
\[
R_{g_{k,r}}(q_k)-R_{g_r}(q) \ge cr^2 
\]
for all $q_k\in M_k$ and all $q \in \{\vert \tilde x\vert \le 5\}\subset M$. 
Then by the previous expansions, we can select $r$ small enough so that 
\[
\left\vert \frac{\vert \mathring A_\Gamma\vert_{g_r}^2}{2} + (\vert \grad H\vert + \bd_\nu H)_{g_r} \right\vert \le \frac{cr^2}{2}
\]
for all $\tau \in (1,3)$ and all points $q$ in $\Gamma = \{\vert \tilde x\vert = \tau\}$.  We fix such a small $r$ and relabel the metric $g_r$ as $g$ and the metrics $g_{k,r}$ as $g_k$.      
Then combining everything we see that there is an $\eps > 0$ so that 
\begin{equation}
\label{eq:zoom}
\frac{R_{g_k}(q_k)}{2} \ge \left(\frac{R_{g}}{2} +\frac{\vert \mathring A\vert^2}{2} + \vert \grad H\vert + \bd_\nu H\right)(q) + \eps
\end{equation}
for all points $q_k \in M_k$ and all points $q\in \{1< d_g(\cdot,p) < 3\} \subset M$. 
We assume that this rescaling has been carried out throughout the remainder of the paper.

\subsection{Transversality}
Let $d(x) = d_g(x,p)$ be the distance function to $p$ in $M$. Let $\Sigma_t = \{d = t\}$ be the geodesic sphere of radius $t$ centered at $p$ in $M$. We want to show that $f_k^{-1}(\Sigma_t)$ is a smooth, closed manifold for almost all choices of $t$. 

\begin{prop}
    The map $f_k$ is transverse to $\Sigma_t$ for almost every $t\in (0,5)$. 
\end{prop}

\begin{proof}
    Consider the composition $d\circ f_k$. This is a smooth map on $f_k^{-1}(M\setminus\{p\})$ and so almost every $t\in (0,5)$ will be a regular value. Select a regular value $t$ of $d\circ f_k$. We claim that for this $t$, the map $f_k$ is transverse to $\Sigma_t$. Indeed, consider some point $x\in \Sigma_t$. Note that $f_k^{-1}(x) \subset (d\circ f_k)^{-1}(t)$ and therefore at each point $y\in f_k^{-1}(x)$ there is some tangent vector $v\in T_yM_k$ for which $D_y (d\circ f_k)(v) \neq 0$. But $D_y (d\circ f_k)(v) \neq 0$ if and only if $D_y f_k(v)$ is not tangential to $\Sigma_t$. Hence the existence of such a $v$ implies that $T_x\Sigma_t \oplus D_yf_k(T_yM_k)$ is the entire tangent space $T_xM$. This shows that $f_k$ is transverse to $\Sigma_t$. 
\end{proof}

After passing to a subsequence which we do not relabel, we can find $c_1\approx 1$ and $c_2\approx 3$ so that $ f_k^{-1}(\{c_1 < d < c_2\})$ is a compact submanifold of $M_k$ with smooth boundary $f_k^{-1}(\Sigma_{c_1})\cup f_k^{-1}(\Sigma_{c_2})$ for all $k$. Since the actual values of $c_1$ and $c_2$ do not materially affect the proof, we shall assume for notational simplicity that 
$c_1=1$ and $c_2=3$ in what follows.

\subsection{Constructing the Model \texorpdfstring{$\mu$}{mu}-Bubble}

We now want to construct a model $\mu$-bubble in the $g$ metric by minimizing the functional 
\[
\mathcal A^h_g(\Omega) = \mathcal H^2_g(\bd^* \Omega) - \int (\chi_\Omega - \chi_{\Omega_0})h\, dv
\]
for a suitable choice of $h$ and $\Omega_0$. 

The first step is to make a good choice of $\Omega_0$.  Define $\Gamma^k_t = f_k^{-1}(\Sigma_t)$ and note that $f_k\colon \Gamma^k_t \to \Sigma_t$ is surjective.  Also, observe the surfaces $\Gamma^k_t$ are the level sets of $d\circ f_k$ and so $\Gamma^k_t$ is a smooth, closed surface in $M_k$ for almost every $t$.

\begin{prop}
    After passing to a subsequence which we do not relabel, we have $\area_{g_k}(\Gamma^k_t) \to \area_g(\Sigma_t)$ for almost every $t\in [1,3]$. 
\end{prop}

\begin{proof} Let $W = \{1 < d < 3\}$ and let $W_k = f_k^{-1}(W)$. 
Applying the co-area formula to $d\circ f_k$, we see that 
\[
\int_{W_k} \vert \grad^{g_k} (d\circ f_k)\vert\, dv_k = \int_1^3 \area_{g_k}(\Gamma^k_t)\, dt. 
\]
Since the composition $d\circ f_k$ is 1-Lipschitz, we have a bound $\vert \grad^{g_k} (d\circ f_k)\vert \le 1$. Thus we deduce that 
\[
\vol_{g_k}(W_k) \ge \int_1^3 \area_{g_k}(\Gamma^k_t)\, dt. 
\]
On the other hand, applying the co-area formula to $d$, we have 
\[
\vol_g(W) = \int_1^3 \area_g(\Sigma_t)\, dt. 
\]
Now since $f_k\colon \Gamma^k_t\to \Sigma_t$ is 1-Lipschitz and surjective, we have $\area_{g_k}(\Gamma^k_t) \ge \area_g(\Sigma_t)$ for almost every $t$. Therefore we obtain 
\[
0 \le \int_1^3 \left\vert \area_{g_k}(\Gamma^k_t) - \area_g(\Sigma_t)\right\vert\, dt \le \vol_{g_k}(W_k) - \vol_g(W), 
\]
and the right hand side goes to 0 by Proposition \ref{prop:subset-volume}.
In particular, we see that $\area_{g_k}(\Gamma^k_t)\to \area_g(\Sigma_t)$ as functions in $L^1([1,3])$.  Hence, after passing to a subsequence, the function $\area_{g_k}(\Gamma^k_t)$ will converge to $\area_g(\Sigma_t)$ pointwise for almost every $t\in [1,3]$. 
\end{proof}

We now select some $\tau$ close to 2 so that $\Gamma^k_\tau$ is a smooth closed manifold for all $k$ and $\area_{g_k}(\Gamma^k_\tau)\to \area_g(\Sigma_\tau)$ as $k\to \infty$.

The next task is to make a good choice of $h$. We will choose $h$ so that $\Sigma_\tau$ is the unique minimizing $\mu$-bubble in the $g$ metric.  We also need to ensure that $h$ satisfies several other properties.  We begin with an elementary lemma. 

\begin{lemma}
\label{Lemma:f}
There is a smooth function $f\colon (1,3)\to \R$ with the following properties: 
\begin{itemize}
\item[(i)] $f(r) \to \infty$ as $r\to 1^+$,
\item[(ii)] $f(r) \to -\infty$ as $r\to 3^-$,
\item[(iii)] $f(r) > 0$ for $r < \tau$,
\item[(iv)] $f(r) < 0$ for $r > \tau$,
\item[(v)] $f(\tau) = 0$ and $f'(\tau) = 0$,
\item[(vi)] $m_1 (r-1)^{-2} \le \vert f(r)\vert \le m_2 (r-1)^{-2}$ as $r\to 1^+$,
\item[(vii)] $\vert f'(r)\vert \le m_3 (r-1)^{-3}$ as $r\to 1^+$,
\item[(viii)] $m_1 (3-r)^{-2} \le \vert f(r)\vert \le m_2 (3-r)^{-2}$ as $r\to 3^-$,
\item[(ix)] $\vert f'(r)\vert \le m_3 (3-r)^{-3}$ as $r\to 3^{-}$. 
\end{itemize}
Here $m_1,m_2,m_3 > 0$ are positive constants. 
\end{lemma}

\begin{proof}
Define the function 
\[
f(r) = \frac{(\tau - r)^3}{(r-1)^2(3-r)^2}. 
\]
Conditions (i), (ii), (iii), and (iv) are immediate.  For (v), we see that $f(\tau) = 0$ and we compute that 
\[
f'(r) = -\frac{(\tau - r)^2 (8 \tau -4\tau r  + r^2 + 4r - 9)}{(r-1)^3 (3-r)^3}
\]
so that $f'(\tau) = 0$ as well.  Finally, as $r\to 1^+$ we see that 
\[
\frac{m_1}{(r-1)^2} \le \vert f(r)\vert  \le \frac{m_2}{(r-1)^2}, \quad \vert f'(r)\vert \le \frac{m_3}{(r-1)^3} 
\]
for some constants $m_1,m_2,m_3 > 0$. Likewise, as $r \to 3^-$ we see that 
\[
\frac{m_1}{(3-r)^2} \le \vert f(r)\vert  \le  \frac{m_2}{(3-r)^2}, \quad \vert f'(r)\vert \le \frac{m_3}{(3-r)^3} 
\]
for some constants $m_1,m_2,m_3 > 0$.  Thus (vi), (vii), (viii), and (ix) also hold.   
\end{proof} 

We now define the function $h$.  Let $H$ be the mean curvature of the foliation $\Sigma_t$, i.e. $H(x)$ is the mean curvature of $\Sigma_{d(x)}$ at the point $x$. 

\begin{prop} There is a function $h$ defined on $\{1 < d < 3\}$ with the following properties: 
\begin{itemize}
    \item[(i)] $h\to \infty$ as $d\to 1^+$,
    \item[(ii)] $h\to -\infty$ as $d\to 3^-$,
    \item[(iii)] $h > H$ inside $\Sigma_\tau$,
    \item[(iv)] $h < H$ outside $\Sigma_\tau$
    \item[(v)] $h = H$ and $\nu(h) = \nu(H)$ on $\Sigma_\tau$
    \item[(vi)] $\frac 3 4 h^2 - \vert \grad h\vert \ge 0$ for $d$ close to 1, and for $d$ close to 3, and near the surface $\Sigma_\tau$. 
\end{itemize}
\end{prop} 

\begin{proof} 
We define $h = H + f(t)$ on $\Sigma_t$. Since $H$ is a smooth bounded function on $\{1 \le d \le 3\}$,  properties (i), (ii), (iii), (iv) are immediate from the corresponding properties of $f$.  It is also immediate that $h = H$ on $\Sigma_\tau$ since $f(\tau) = 0$, and that $\nu(h) = \nu(H)$ since $f'(\tau) = 0$ and the foliation $\Sigma_t$ consists of level sets of the distance function. 

It remains to verify (vi).  Near the surface $\Sigma_\tau$ we know that $H$ is very nearly $\frac{2}{\tau}$ and $\vert \grad H\vert$ is very nearly $\frac{2}{\tau^2}$.  Since $f(\tau) = f'(\tau) = 0$ and $\tau$ is close to 2, it follows that we have 
\[
\frac 3 4 h^2 - \vert \grad h\vert \ge 0
\]
near $\Sigma_\tau$. Finally, note that 
\begin{align*}
\frac 3 4 h^2 - \vert \grad h\vert &= \frac{3}{4} (H + f)^2 - \vert \grad H + \grad f\vert \\
&\ge \frac 3 4 f^2 + \frac 3 2 f H + \frac 3 4 H^2 - \vert \grad H\vert - \vert \grad f\vert \\
&\ge \frac 3 4 f^2 - C(1+f) - \vert \grad f\vert.
\end{align*} 
Hence using properties (vi), (vii), (viii), and (ix) of $f$, we deduce that $\frac 3 4 h^2 - \vert \grad h\vert \ge 0$ when $d$ is close to $1$ and when $d$ is close to $3$. 
\end{proof}

For simplicity, let us introduce the notation $\Sigma = \Sigma_\tau$, and define $\Omega_0 = \{d < \tau\}$ to be the set enclosed by $\Sigma_\tau$.  Let $\mathcal C$ be the class of Caccioppoli sets in $M$ whose symmetric difference with $\Omega_0$ is compactly contained in $\{1 < d < 3\}$.  For the metric $g$, we define the following $\mu$-bubble functional:
\[
\mathcal A^h_{g}(\Omega) = \mathcal H^2_{g}(\bd^* \Omega) - \int (\chi_\Omega - \chi_{\Omega_0}) h\, dv, \quad \Omega\in \mathcal C. 
\]
We now confirm that $\Omega_0$ is the unique minimizer of $\mathcal A^h_{g}$.  

\begin{prop} \label{Proposition:QuantitativeMinimality}
The set $\Omega_0$ is the unique minimizer of $\mathcal A^h_g$. In fact, let $U$ be any small neighborhood of $\Sigma$. Then there is a constant $c = c(U) > 0$ such that 
\[
\mathcal A^h_g(\Omega) \ge \mathcal A^h_g(\Omega_0) + c \vol( (\Omega \operatorname{\Delta} \Omega_0) - U)
\]
for all $\Omega \in \mathcal C$. 
\end{prop}

\begin{proof}
This follows from the divergence theorem applied to the unit normal $\nu$ to the spheres $\Sigma_t$, noting that 
\[
c + (\chi_\Omega - \chi_{\Omega_0}) h  \le (\chi_\Omega - \chi_{\Omega_0})H 
\]
for some $c = c(U) > 0$ at all points $x \in (\Omega \operatorname{\Delta} \Omega_0) - U$ by properties (iii) and (iv) of $h$. 
\end{proof} 

\begin{prop}
\label{Proposition:DegenerateStable}
The set $\Omega_0$ is a degenerate stable critical point for $\mathcal A^h_g$ with first eigenfunction equal to 1. 
\end{prop}

\begin{proof} First consider the functional 
\[
\mathcal A^H_g(\Omega) = \mathcal H^2_g(\bd^* \Omega) - \int (\chi_\Omega - \chi_{\Omega_0}) H\, dv. 
\]
Let $\Omega_t$ be the region enclosed by $\Sigma_t$.  Since by definition $H$ is the mean curvature of $\Sigma_t$, we see that each $\Omega_t$ is a critical point of $A^H_g$.  Since $\Sigma_t$ is a foliation by level sets of the distance function, it follows that each $\Omega_t$ is a degenerate stable critical point of $\mathcal A^H_g$ and that the function 1 achieves equality in the stability inequality for $\mathcal A^H_g$ on $\Sigma_t$.  In particular, taking $t = \tau$ and plugging 1 into the stability inequality we get 
\[
\int_\Sigma \frac{R_g}{2} + \frac{\vert \mathring A\vert^2}{2} + \frac 3 4 H^2 + \nu(H)\, da = 4\pi. 
\]
Now by construction we have $h = H$ and $\nu(h) = \nu(H)$ on $\Sigma$. Therefore we also have 
\[
\int_\Sigma \frac{R_g}{2} + \frac{\vert \mathring A\vert^2}{2} + \frac 3 4 h^2 + \nu(h)\, da = 4\pi. 
\]
This implies that $\Omega_0$ is a degenerate stable critical point for $\mathcal A^h_g$ with first eigenfunction equal to 1. 
\end{proof}

\subsection{The Approximating \texorpdfstring{$\mu$}{mu}-Bubbles} 

We now define analogous $\mu$-bubble functionals in $M_k$. 
We define functions $h_k$ on $f_k^{-1}(\{1<d<3\})$ by setting $h_k = h\circ f_k$ and consider the functionals $\mathcal A^{h_k}_{g_k}$ defined by 
\[
\mathcal A^{h_k}_{g_k}(\Omega) = \mathcal H^2_{g_k}(\bd^* \Omega) -\int (\chi_\Omega - \chi_{\Omega^k_0})h\, dv_k, \quad \Omega\in \mathcal C_k.
\]
Here $\Omega^k_0 = f_k^{-1}(\Omega_0)$ is a Caccioppoli set in $M_k$ by our choice of $\Omega_0$, and $\mathcal C_k$ is the set of all Caccioppoli sets in $M_k$ whose symmetric difference with $\Omega^k_0$ is compactly supported in $f_k^{-1}(\{1<d<3\})$. 
Since $f_k^{-1}(\{1<d<3\})$ is a manifold with boundary, and $\Omega^k_0$ is a Caccioppoli set containing a neighborhood of $f_k^{-1}(\Sigma_1)$ and disjoint from a neighborhood of $f_k^{-1}(\Sigma_3)$, and $h_k \to \infty$ near $f_k^{-1}(\Sigma_1)$ and $h_k\to -\infty$ near $f_k^{-1}(\Sigma_3)$ there will be a minimizer $\Omega_k$ for this functional by Proposition \ref{prop: mu-bubble regularity}.

\subsection{Convergence}

In this section, we show that the $\mu$-bubbles $\Omega_k$ converge to $\Omega$ in a suitable weak sense. We note that $\bd \Omega_k$ and $\Sigma$ are smooth surfaces, and so we shall write $\area_{g_k}(\bd \Omega_k)$ and $\area_g(\Sigma)$ accordingly. 

\begin{prop}
    We have $\mathcal A^{h_k}_{g_k}(\Omega_k) \to \mathcal A^h_g(\Omega_0)$ and $\mathcal A^h_g((f_k)_\sharp \Omega_k) \to \mathcal A^h_g(\Omega_0)$. 
\end{prop}

\begin{proof}
    First, note that $\Omega^k_0$ is a valid competitor for $\mathcal A^{h_k}_{g_k}$ and that $\bd \Omega^k_0 = \Gamma^k_\tau$.  Hence we have  
    \[
    \mathcal A^{h_k}_{g_k}(\Omega_k) \le \mathcal A^{h_k}_{g_k}(\Omega^k_0) = \area_{g_k}(\Gamma^k_\tau) \to \area_g(\Sigma) 
    \]
    as $k\to \infty$.  

    Now consider the pushforward  $(f_k)_\sharp \Omega_k$. Since $f_k:M_k\to M$ has degree one modulo two, by definition of degree of continuous mapping, we have  
    \begin{align}\label{eq: degree one map}
    (f_k)_{\#}\big(f_k^{-1}(\Omega_t)\big)=\Omega_t \quad \text{as Caccioppoli sets}
    \end{align}
    for almost every $t\in[1,3]$.
    Since $\Omega_k$ is a valid competitor for $\mathcal A_{g_k}^{h_k}$, it contains a neighborhood of $f_k^{-1}(\Sigma_1)$ and hence it contains $f_k^{-1}(\Omega_{1+\delta})$ for which $\eqref{eq: degree one map}$ holds. 
    It follows that $(f_k)_{\#}\Omega_k$ is a valid competitor for $\mathcal A_g^h$, since it contains a neighborhood of $\Sigma_1$ and disjoint from a neighborhood of $\Sigma_3$.

Recall that $\Omega^k_0 = f_k^{-1}(\Omega_0)$ and that $(f_k)_\sharp \Omega^k_0 = \Omega_0$.  Since $\Omega_0$ minimizes $\mathcal A^h_g$, we deduce that
\begin{align*}
\area_g(\Sigma) = \mathcal A^h_g(\Omega_0) &\le \mathcal A^h_g((f_k)_\sharp \Omega_k) \phantom{\int}\\
        &= \mathcal H^2_g(\bd^*(f_k)_\sharp  \Omega_k) - \int_M (\chi_{(f_k)_\sharp \Omega_k} - \chi_{\Omega_0})h\, dv\\
        &\le \area_{g_k}(\bd \Omega_k) - \int_{M} (\chi_{(f_k)_\sharp \Omega_k} - \chi_{(f_k)_\sharp \Omega_0^k})h\, dv\\
        &= \mathcal A^{h_k}_{g_k}(\Omega_k) + \int_{M_k} (\chi_{\Omega_k}-\chi_{\Omega_0^k}) (h\circ f_k)\, dv_k - \int_M (\chi_{(f_k)_\sharp \Omega_k} - \chi_{(f_k)_\sharp \Omega_0^k})h\, dv.
        \end{align*}
Here we used Proposition \ref{prop:area-ineq} to get from the second to the third line. 
Next, we examine the bulk difference in more detail.  Using Proposition \ref{prop:integral-formula}, we rewrite this as 
\begin{align*}
&I_k := \int_{M_k} (\chi_{\Omega_k}-\chi_{\Omega_0^k}) (h\circ f_k)\, dv_k - \int_M (\chi_{(f_k)_\sharp \Omega_k} - \chi_{(f_k)_\sharp \Omega_0^k})h\, dv \\
&\qquad = \int_{M_k} (\chi_{\Omega_k} - \chi_{\Omega^k_0}) (h\circ f_k) (1-J_f)\, dv_{k} + \int_{M} ((\theta_{\Omega_k} - \overline \theta_{\Omega_k}) - (\theta_{\Omega^k_0} - \overline \theta_{\Omega^k_0}))h\, dv. 
\end{align*} 
By the previous calculation and the fact that $\limsup \mathcal A^{h_k}_{g_k}(\Omega_k) \le \area_g(\Sigma)$, we must have $\liminf I_k \ge 0$. 

On the other hand, we claim that $\limsup I_k \le 0$.   Split $\{1 < d < 3\}$ into two regions.  Define $A = \{a < d < b\}$ where $a > 1$ and $b < 3$ are chosen so that $h > 0$ inside $\Sigma_a$ and $h < 0$ outside $\Sigma_b$.  Then we have 
\[
\int_{f_k^{-1}(A)} (\chi_{\Omega_k} - \chi_{\Omega^k_0}) (h\circ f_k) (1-J_{f_k})\, dv_k \to 0 
\]
since $h\circ f_k$ is bounded on $f_k^{-1}(A)$ and $\vol(M_k,g_k) \to \vol(M,g)$.  On the other hand, since $f_k$ is 1-Lipschitz, we have $1-J_{f_k} \ge 0$ and we see that 
\[
(\chi_{\Omega_k} - \chi_{\Omega_0^k}) (h\circ f_k) \le 0
\]
outside $f_k^{-1}(A)$. Thus we have 
\[
\int_{M_k - f_k^{-1}(A)} (\chi_{\Omega_k} - \chi_{\Omega^k_0})(h\circ f_k)(1-J_{f_k})\, dv_k \le 0. 
\]
Now we examine the multiplicity term. Again we split into an integral over $A$ and the remainder. On $A$, we have 
\[
\int_{A} ((\theta_{\Omega_k} - \overline \theta_{\Omega_k}) - (\theta_{\Omega^k_0} - \overline \theta_{\Omega^k_0}))h\, dv = \int_A (\theta_{\Omega_k} - \overline \theta_{\Omega_k}) h\, dv - \int_A (\theta_{\Omega^k_0} - \overline \theta_{\Omega^k_0})h\, dv. 
\]
Since $h$ is bounded on $A$, we can use  Proposition \ref{prop:multiplicity-bound} and Proposition \ref{prop:multiplicity} to estimate 
\begin{gather*}
\left\vert \int_A (\theta_{\Omega_k} - \overline \theta_{\Omega_k})h\, dv\right\vert \le \sup_A \vert h\vert \int_A (\theta_{M_k} - \overline \theta_{M_k})\, dv \to 0,\\
\left\vert \int_A (\theta_{\Omega^k_0} - \overline \theta_{\Omega^k_0})h\, dv\right\vert \le \sup_A \vert h\vert \int_A (\theta_{M_k} - \overline \theta_{M_k})\, dv \to 0.
\end{gather*}
It remains to understand what happens outside of $A$. We claim that 
\[
((\theta_{\Omega_k} - \overline \theta_{\Omega_k}) - (\theta_{\Omega^k_0} - \overline \theta_{\Omega^k_0}))h \le 0 
\]
outside $A$. First consider a point $x$ lying outside $\Sigma_b$.  Then we have $\theta_{\Omega_k} - \overline \theta_{\Omega_k} \ge 0$ and $\theta_{\Omega^k_0} = \overline \theta_{\Omega^k_0} = 0$ and $h \le 0$ so the sign is as claimed.  Second, consider a point $x$ lying inside $\Sigma_a$.  Then $x \in \Omega_0$ and all the preimages of $x$ lie in $\Omega^k_0$.  We can suppose there are only finitely many preimages $y_1,\hdots,y_\ell \in \Omega_{0}^k$ and we note that $\ell$ must be odd. Let us suppose that $j$ out of these $\ell$ points belong to $\Omega_k$.  Then 
\begin{align*}
(\theta_{\Omega_k} - \overline \theta_{\Omega_k}) - (\theta_{\Omega^k_0} - \overline \theta_{\Omega^k_0}) & = (j - j\text{ mod } 2) - (\ell -1) \\
&= (j - \ell) + (1- j\text{ mod } 2). 
\end{align*} 
If $j$ is odd, then we have $(j-\ell) + (1-j\text{ mod } 2) = j-\ell \le 0$. On the other hand, if $j$ is even then $j +1 \le \ell$ since $\ell$ is odd, and again we conclude that $(j-\ell) + (1-j\text{ mod } 2) = j-\ell + 1 \le 0$.  Since $h > 0$ inside $\Sigma_a$, the sign is as claimed inside $\Sigma_a$ as well.  This verifies the claim, and it follows that 
\[
\int_{M-A} ((\theta_{\Omega_k} - \overline \theta_{\Omega_k}) - (\theta_{\Omega^k_0} - \overline \theta_{\Omega^k_0}))h\, dv \le 0. 
\]
Combining these observations, we deduce that $\limsup I_k \le 0$. 

Since we have already shown that $\liminf I_k \ge 0$, it follows that $I_k \to 0$ as $k\to \infty$.  Recall from earlier that we have shown $\limsup \mathcal A^h_{g_k}(\Omega_k) \le \mathcal A^h_g(\Omega_0)$ and that 
\[
\mathcal A^h_g(\Omega_0) \le \mathcal A^h_g((f_k)_\sharp \Omega_k) = \mathcal A^{h_k}_{g_k}(\Omega_k) + I_k.
\]
Since $I_k\to 0$, we can then deduce that $\mathcal A^{h_k}_{g_k}(\Omega_k) \to \mathcal A^h_g(\Omega_0)$ and $\mathcal A^h_g((f_k)_\sharp \Omega_k) \to \mathcal A^h_g(\Omega_0)$. 
\end{proof}

\begin{prop}
    We have $(f_k)_\sharp \Omega_k \to \Omega_0$ as Caccioppoli sets and $\vert \bd  (f_k)_\sharp \Omega_k\vert  \to \vert \bd \Omega_0\vert$ as varifolds in $\{d\le 5\}$. 
\end{prop}

\begin{proof}
    The convergence as Caccioppoli sets follows from Proposition \ref{Proposition:QuantitativeMinimality} since $\mathcal A^h_g((f_k)_\sharp \Omega_k) \to \mathcal A^h_g(\Omega_0)$ and $\Omega_0$ is the unique minimizer for $\mathcal A^h_g$.
    It follows that $\bd \lb (f_k)_\sharp \Omega_k\rb \to \bd \lb \Omega_0\rb$ as currents. 
    Then by Proposition \ref{prop:area-converges}, to get convergence as varifolds, it suffices to show that $\M(\bd \lb (f_k)_\sharp \Omega_k\rb)  \to \area_g(\Sigma)$. Observe that 
    \[
    \mathcal A^h_g((f_k)_\sharp \Omega_k) = \M(\bd \lb (f_k)_\sharp \Omega_k\rb ) - \int (\chi_{(f_k)_\sharp \Omega_k} - \chi_{\Omega_0})h\, dv \to \mathcal A^h_g(\Omega_0) = \area_g(\Sigma). 
    \]
    Since $\liminf \M(\bd \lb (f_k)_\sharp \Omega_k\rb ) \ge \area_g(\Sigma)$, it is enough to show that 
    \[
    \limsup \int (\chi_{(f_k)_\sharp \Omega_k}-\chi_{\Omega_0})h\, dv \le 0. 
    \]
    To that end, again consider the annulus $A$ from before. We have 
    \begin{align*} 
    \limsup \int (\chi_{(f_k)_\sharp \Omega_k}-\chi_{\Omega_0})h\, dv & = \int (\chi_{(f_k)_\sharp \Omega_k \cap A}-\chi_{\Omega_0 \cap A})h\, dv\\ 
    &\qquad + \int (\chi_{(f_k)_\sharp \Omega_k-A}-\chi_{\Omega_0-A})h\, dv.
    \end{align*} 
    The first integral on the right goes to 0 since $(f_k)_\sharp \Omega_k \to \Omega_0$ as Caccioppoli sets and $h$ is bounded on $A$.  We claim that the second integral is non-positive. Indeed, this follows from the fact that $\chi_{\Omega_0-A} = 1$ and $h > 0$ inside $\Sigma_a$, and $\chi_{\Omega_0-A} = 0$ and $h < 0$ outside $\Sigma_b$.  Thus we indeed have 
\[
    \limsup \int (\chi_{(f_k)_\sharp \Omega_k}-\chi_{\Omega_0})h\, dv \le 0. 
    \]
    Hence $\M(\bd \lb (f_k)_\sharp \Omega_k\rb ) \to \area_g(\Sigma)$, and we obtain the varifold convergence. We note for future use that we have also demonstrated that 
    \[
    \int (\chi_{(f_k)_\sharp \Omega_k}-\chi_{\Omega_0})h\, dv \to 0
    \]
    as well. 
\end{proof}

\begin{prop}
    We have $\area_{g_k}(\bd \Omega_k)\to \area_g(\Sigma)$. 
\end{prop}

\begin{proof}
    We know that $\mathcal A^{h_k}_{g_k}(\Omega_k) \to \mathcal A^h_g(\Omega_0) = \area_g(\Sigma)$, and we have 
    \[
    \mathcal A^{h_k}_{g_k}(\Omega_k) = \area_{g_k}(\bd \Omega_k) - \int (\chi_{\Omega_k}-\chi_{\Omega^k_0}) (h\circ f_k)\, dv_k. 
    \]
    Hence it is equivalent to show that 
    \[
    \int (\chi_{\Omega_k}-\chi_{\Omega^k_0}) (h\circ f_k)\, dv_k \to 0
    \]
    as $k\to \infty$. We rewrite this as 
    \[
    \int (\chi_{\Omega_k}-\chi_{\Omega^k_0}) (h\circ f_k)\, dv_k = I_k + \int (\chi_{(f_k)_\sharp \Omega_k}-\chi_{\Omega_0})h\, dv. 
    \]
    We have already shown that both terms on the right go to zero and this concludes the proof. 
\end{proof}

\begin{prop}
    Let $U$ be a small tubular neighborhood of $\Sigma$. Then we have $\area_{g_k}(\bd \Omega_k - f_k^{-1}(U)) \to 0$ as $k\to \infty$. 
\end{prop}

\begin{proof}
Suppose for contradiction that $\liminf \area_{g_k}(\bd \Omega_k - f_k^{-1}(U)) \ge 2\eps > 0$. Since we have already seen that $\area_{g_k}(\bd \Omega_k) \to \area_g(\Sigma)$ as $k\to \infty$, this implies that 
\[
\area_{g_k}(\bd \Omega_k \cap f_k^{-1}(U)) \le \area_g(\Sigma) - \eps
\]
for large $k$. But then  
\[
\| \bd (f_k)_\sharp \lb \Omega_k \rb \|(U) \le \area_{g_k}(\bd \Omega_k \cap f_k^{-1}(U)) \le \area_g(\Sigma) - \eps
\]
and this contradicts the varifold convergence $\vert \bd (f_k)_\sharp \Omega_k\vert \to \vert \bd \Omega_0\vert$. 
\end{proof}

\begin{prop}
    There is a connected component $\Gamma_k$ of $\bd \Omega_k$ such that $\vert (f_k)_\sharp \lb \Gamma_k\rb \vert\to \vert \Sigma\vert$ as varifolds. 
\end{prop}

\begin{proof}
    We know that $\bd \lb (f_k)_\sharp \Omega_k\rb  = (f_k)_\sharp \lb \bd \Omega_k\rb $ converges to $\lb \Sigma\rb$ as currents. Let $\Gamma_k$ be the component of $\bd \Omega_k$ for which $(f_k)_\sharp \lb \Gamma_k\rb$ has the largest mass in $M$. By the compactness theorem, after passing to a subsequence, $(f_k)_\sharp \lb \Gamma_k\rb $ will converge to some limiting current $T$.  Moreover, since $\vert (f_k)_\sharp \lb \bd \Omega_k\rb\vert$ converges to $\vert \Sigma\vert$ as varifolds, the current $T$ is necessarily supported in $\Sigma$. Then by the constancy theorem either $T = \lb \Sigma\rb $ or $T = 0$. We claim that $T = \lb \Sigma\rb$. 

    Suppose to the contrary that $T = 0$. Then we consider two subcases and show that both lead to contradiction. First suppose that $\liminf \M((f_k)_\sharp \lb \Gamma_k\rb ) > 0$. Then we let $\Lambda_k$ be the union of all remaining components of $\bd \Omega_k$ and note that $\limsup \M((f_k)_\sharp \lb \Lambda_k\rb ) < \area(\Sigma)$. After passing to a subsequence $(f_k)_\sharp \lb \Lambda_k\rb $ will converge to a limiting current $S$.  By the same argument as before, $S$ is supported in $\Sigma$ and so by the constancy theorem either $S = 0$ or $S = \lb \Sigma \rb$. Since the mass of $(f_k)_\sharp \lb \Lambda_k\rb$ is always a definite amount less than $\area_g(\Sigma)$, we must have $S = 0$. But then since $(f_k)_\sharp \lb \Gamma_k\rb \to 0$ and $(f_k)_\sharp \lb \bd \Omega_k\rb  = (f_k)_\sharp \lb \Gamma_k\rb  + (f_k)_\sharp \lb \Lambda_k\rb $, this implies $(f_k)_\sharp \lb \bd \Omega_k\rb \to 0$ which is a contradiction. 

    Otherwise, we have $\M((f_k)_\sharp \lb \Gamma_k \rb)\to 0$ and we conclude that the maximum of $\M((f_k)_\sharp \lb \Lambda\rb)$ over all components $\Lambda$ of $\bd \Omega_k$ goes to 0. Thus we can split $\bd \Omega_k$ into $\Lambda^k_1 \cup \Lambda_k^2$ where each $\Lambda_k^i$ is a union of components of $\bd \Omega_k$ for which $\M((f_k)_\sharp \lb \Lambda_k^i\rb)$ is at most  $\frac 1 2 \M((f_k)_\sharp \lb \bd \Omega_k\rb) + \eps$. After passing to a subsequence, we have $\Lambda_k^1 \to S_1$ and $\Lambda_k^2\to S_2$ as currents and by the same argument as before we must have $S_1=0$ and $S_2 = 0$. Again this contradicts that $(f_k)_\sharp\lb \bd \Omega_k\rb$ converges to $\lb \Sigma\rb $. 

    Thus we deduce that $(f_k)_\sharp\lb  \Gamma_k\rb$ converges to $\lb \Sigma\rb$ as currents. This implies that 
    \[
    \liminf \M((f_k)_\sharp \lb \Gamma_k\rb) \ge \area_g(\Sigma).
    \]
    But we also know that 
    \[
    \M((f_k)_\sharp \lb \Gamma_k\rb) \le \area_{g_k}(\Gamma_k) \le \area_{g_k}(\bd \Omega_k)\to \area_g(\Sigma)
    \]
    and hence $\M((f_k)_\sharp \lb \Gamma_k\rb)\to \area_g(\Sigma)$. This implies that $\vert (f_k)_\sharp \lb \Gamma_k\rb \vert\to \vert \Sigma\vert$ as varifolds, and the proposition is proved. 
\end{proof}

\subsection{Comparing Stability Inequalities} Finally, we can prove the main result. The idea is to compare the stability inequality on $\Gamma_k$ with the stability inequality on $\Sigma$. Since $\Sigma$ is a degenerate stable critical point for $\mathcal A^h_g$ with first eigenfunction 1 according to Proposition \ref{Proposition:DegenerateStable}, we have 
\[
\int_\Sigma \frac {R_g}{2} + \frac{\| \mathring A\|^2}{2} + \big(\vert \grad h\vert + \nu(h)\big) + \frac 3 4 h^2 - \vert \grad h\vert\, da = 4\pi. 
\]
Likewise, we can apply the stability inequality on $\Gamma_k$ with the test function 1 and use the fact that $\Gamma_k$ is connected to deduce that 
\[
\int_{\Gamma_k} \frac{R_{g_k}}{2} + \frac 3 4 (h\circ f_k)^2 + \nu_k(h\circ f_k)\, da_k \le 4\pi. 
\]
By our scaling procedure \eqref{eq:zoom} and the fact that $\area_{g_k}(\Gamma_k) \to \area_g(\Sigma)$, we see that 
\[
\left[ \int_{\Gamma_k} \frac{R_{g_k}}{2} da_k\right] - \left[\int_\Sigma \frac{R_g}{2} + \frac{\| \mathring A\|^2}{2} + \big(\vert \grad h\vert + \nu(h)\big)\, da\right] \ge c > 0
\]
where $c$ does not depend on $k$. Hence to get a contradiction it suffices to show that 
\[
\liminf \int_{\Gamma_k} \frac 3 4 (h\circ f_k)^2 + \nu_k(h\circ f_k)\, da_k \ge \int_\Sigma \frac 3 4 h^2 - \vert \grad h\vert\, da. 
\]
Observe that 
\begin{align*}
\int_{\Gamma_k} \frac 3 4 (h\circ f_k)^2 + \nu_k(h\circ f_k)\, da_k &\ge \int_{\Gamma_k} \frac 3 4 (h\circ f_k)^2 - \vert \grad^{g_k}(h\circ f_k)\vert\, da_k \\
&\ge \int_{\Gamma_k} \frac 3 4 h(f_k(x))^2 - \vert \grad^g h\vert(f_k(x))\, da_k,
\end{align*}
where we used the fact that $f_k$ is 1-Lipschitz to get from the first to the second line. Hence we just need to show that 
\[
\liminf \int_{\Gamma_k} \left(\frac 3 4 h^2 - \vert \grad h\vert\right)\circ f_k\, da_k \ge \int_\Sigma \frac 3 4 h^2 - \vert \grad h\vert\, da. 
\]
To see this, we split the integral into three pieces. Let $U$ be a small tubular neighborhood of $\Sigma$ on which 
\[
\frac 3 4 h^2 - \vert \grad h\vert \ge 0. 
\]
Let $V = \{1<d<1+\eta\}\cup \{3-\eta < d < 3\}$ where $\eta$ is chosen so that 
\[
\frac 3 4 h^2 - \vert \grad h\vert \ge 0 
\]
on $V$. 
Such choices of $U$ and $V$ are possible by the properties of $h$. Finally, let $W = \{1<d<3\}-(U\cup V)$. 

We make the following observations. First, since $\frac 3 4 h^2 - \vert \grad h\vert$ is bounded on $W$ and $\area_{g_k}(\Gamma_k - f_k^{-1}(W))\to 0$, we have 
\[
\int_{\Gamma_k \cap f_k^{-1}(W)} \left(\frac 3 4 h^2 - \vert \grad h\vert\right)\circ f_k\, da_k \to 0. 
\]
Second, by the choice of $V$, we have 
\[
\int_{\Gamma_k \cap f_k^{-1}(V)} \left(\frac 3 4 h^2 - \vert \grad h\vert\right)\circ f_k\, da_k \ge 0. 
\]
Third since $\frac 3 4 h^2 - \vert \grad h\vert \ge 0$ on $U$ and $f_k$ is 1-Lipschitz, we have 
\begin{align*}
    \int_{\Gamma_k \cap f_k^{-1}(U)} \left(\frac 3 4 h^2 - \vert \grad h\vert\right)\circ f_k \, da_k \ge \int_{\vert (f_k)_\sharp \lb \Gamma_k\rb \vert \cap U} \frac 3 4 h^2 - \vert \grad h\vert\, da. 
\end{align*}
Then by the varifold convergence $\vert (f_k)_\sharp \lb \Gamma_k\rb \vert \to \vert \Sigma\vert$, we have 
\[
\int_{\vert(f_k)_\sharp \lb \Gamma_k\rb \vert \cap U} \frac 3 4 h^2 - \vert \grad h\vert^2\, da \to \int_\Sigma \frac 3 4 h^2 - \vert \grad h\vert^2\, da. 
\]
Combining everything, we deduce that 
\[
\liminf \int_{\Gamma_k} \left(\frac 3 4 h^2 - \vert \grad h\vert\right)\circ f_k\, da_k \ge \int_\Sigma \frac 3 4 h^2 - \vert \grad h\vert\, da.
\]
This is a contradiction and the theorem follows.

\bibliographystyle{plain}
\bibliography{bibliography}

\end{document}